\documentclass[a4paper, oneside, notitlepage, 11pt, reqno]{amsart}
\usepackage{amsmath,amsfonts,amssymb}
\usepackage{amsthm} 
\usepackage{graphicx} 
\usepackage{subfig}
\usepackage{listings}
\usepackage{geometry}
\usepackage{hyperref}
\usepackage{wrapfig}
\usepackage{url}
\usepackage{tikz}
\usetikzlibrary{shapes, positioning, arrows}
\usepackage{eso-pic}
\usepackage{mathrsfs}
\usepackage{xcolor}
\usepackage{mathtools}

\newcommand{\ds}{\, \mathrm{d}s}

\newcommand{\dx}{\, \mathrm{d}x}
\newcommand{\dy}{\, \mathrm{d}y}
\newcommand{\dz}{\, \mathrm{d}z}
\newcommand{\dxi}{\, \mathrm{d}\xi}

\newcommand{\R}{\mathbb{R}}
\newcommand{\Z}{\mathbb{Z}}
\newcommand{\N}{\mathbb{N}}
\newcommand{\C}{\mathbb{C}}

\DeclareMathOperator{\F}{{\mathscr F}}

\DeclareMathOperator{\sign}{sign}

\newtheorem{theorem}{Theorem}[section]
\newtheorem{lemma}[theorem]{Lemma}
\newtheorem{corollary}[theorem]{Corollary}

\numberwithin{equation}{section}
\theoremstyle{remark}
\newtheorem{remark}[theorem]{Remark}

\title{Decay of solitary waves}
\author{Mathias Nikolai Arnesen}
\date{}

\begin{document}

\begin{abstract}
In this paper we consider the decay rate of solitary-wave solutions to some classes of non-linear and non-local dispersive equations, including for example the Whitham equation and a Whitham--Boussinesq system. The dispersive term is represented by a Fourier multiplier operator that has a real analytic symbol, and we show that all supercritical solitary-wave solutions decay exponentially, and moreover provide the exact decay rate, which in general will depend on the speed of the wave. We also prove that solitary waves have only one crest and are symmetric for some class of equations.
\end{abstract}

\maketitle

\section{Introduction}
%This paper is devoted to the study of the asymptotic behaviour of solitary-wave solutions to some types of non-linear dispersive equations (and systems of equations), in particular model equations for water waves. Consider for instance the non-linear dispersive equation 
%\begin{equation}
%\label{eq: example kdv type}
%u_t+L(u)_x+G(u)_x=0,
%\end{equation}
%where $L$ is a Fourier multiplier operator with symbol $m \colon \R \rightarrow \R$, meaning that
%\begin{equation*}
%\widehat{L \varphi}(\xi)=m(\xi)\widehat{\varphi}(\xi),
%\end{equation*}
%and $G$ is a non-linear term. This is a typical form of model equations for the water-wave problem, and many of the most prominent models can be cast in this form. For instance, if $G(u)=u^2$ and $m(\xi)=1-\xi^2$ we get the Korteweg-de Vries equation, and if $m(\xi)=\sqrt{\frac{\tanh(\xi)}{\xi}}$ we get the Whitham equation \cite{Whitham1967vma}, to mention a few. Assuming that $u$ is a solitary-wave solution to \eqref{eq: example kdv type} moving to the right with speed $c$, that is, $u(x,t)=u(x-ct)$ and $\lim_{x-ct\rightarrow \pm \infty} u(x-ct)=0$, we can integrate \eqref{eq: example kdv type} and get that $u$ solves the equation
%\begin{equation*}
%	cu - L(u) - G(u)=0.
%\end{equation*}
%

This paper is devoted to the study of decaying solutions $u\colon \R \rightarrow \R$ to equations of the form
\begin{equation}
	\label{eq: main}
	cu-L(u)-G_c(u)=0,
\end{equation}
where $c>0$ is a parameter, $L$ is a Fourier multiplier operator with symbol $m \colon \R \rightarrow \R$, meaning that
\begin{equation*}
\widehat{L \varphi}(\xi)=m(\xi)\widehat{\varphi}(\xi),
\end{equation*}
and $G_c$ is some non-linear function that may depend on the parameter $c$ (see examples and assumptions below). In general we will suppress the potential dependency of $G_c$ on $c$ and simply write $G$. By decaying solutions, we mean that
\[
\lim_{|x|\rightarrow  \infty} u(x)=0.
\]
The goal of this paper is to determine the rate of decay of the solutions, under some assumptions on $m$ and $G$.

Equations of the form \eqref{eq: main} are of interest as solitary-wave solutions to a wide variety of model equations for water waves can be represented as decaying solutions to \eqref{eq: main}. Consider for instance the non-linear dispersive equation 
\begin{equation}
\label{eq: example kdv type}
u_t+L(u)_x+G(u)_x=0.
\end{equation}
This is a typical form of model equations for the water-wave problem, and many of the most prominent models can be cast in this form. For instance, if $G(u)=u^2$ and $m(\xi)=1-\xi^2$ we get the Korteweg-de Vries equation, and if $m(\xi)=\sqrt{\frac{\tanh(\xi)}{\xi}}$ we get the Whitham equation \cite{Whitham1967vma}, to mention a few. Assuming that $u$ is a solitary-wave solution to \eqref{eq: example kdv type} moving to the right with speed $c$, that is, $u(x,t)=u(x-ct)$ and $\lim_{|x-ct|\rightarrow \infty} u(x-ct)=0$, we can integrate \eqref{eq: example kdv type} to get \eqref{eq: main}.

Another example is solitary-wave solutions to the Whitham-Boussinesq type system:
\begin{align}
	\eta_t= & -L(u)_x-(\eta u)_x \label{eq: whitham-boussinesq} \\ 
	u_t = & - \eta_x -uu_x, \nonumber
\end{align}
where $\eta$ is the surface elevation and $u$ is the velocity at the surface in the rightwards direction. A solitary-wave solution to \eqref{eq: whitham-boussinesq} with speed $c>0$ is a solution of the form $\eta(x,t)=\eta(x-ct)$, $u(x,t)=u(x-ct)$ such that $u(\zeta), \eta(\zeta)\rightarrow 0$ as $|\zeta|\rightarrow \infty$. Under this ansatz, one finds that (see \cite{Nilsson2018sws}) $\eta=u(c-\frac{u}{2})$ and
\begin{equation}
\label{eq: main eq}
L(u)-u(u-c)\left(\frac{u}{2}-c\right)=0.
\end{equation}
This can be written in the form \eqref{eq: main} with $G(u)=\frac{u^2}{2}(3c-u)$ and $c$ replaced by $c^2$ in the first term.

We will not concern ourselves with existence theory in this paper, but simply establish the decay properties of solutions, should they exist. For results on existence of solitary-wave solutions to equations of the form \eqref{eq: example kdv type}, see for instance \cite{EGW} for weak dispersion (i.e. when $L$ is a smoothing operator), and \cite{Arnesen2016eos} for when $L$ is a differentiating operator.

If $m(\xi)\neq c$ for all $\xi\in \R$, we can formally write \eqref{eq: main} as
\begin{equation}
\label{eq: mellomregning main}
	u=K_c \ast G(u), \,\, K_c=\F^{-1}\left( \frac{1}{c-m}\right).
\end{equation}
However, if $m$ decays (that is, $L$ is a smoothing operator), then $(c-m)^{-1}$ tends to $c^{-1}>0$ at infinity, and $K_c$ exists only in a distributional sense. To remedy this, we can apply the operator $L$ to both sides of \eqref{eq: mellomregning main}, and from \eqref{eq: main} we get
\begin{equation}
	\label{eq: main eq general}
	u\left(c-\frac{G(u)}{u}\right)=H_c\ast G(u),
\end{equation}
where
\begin{equation}
\label{eq: H_c}
H_c=\F^{-1}\left(\frac{m}{c-m}\right).
\end{equation}

We will work under the assumption that $L$ is a smoothing operator (see assumption (A1) below), but the results can be applied to differentiating operators as well, if the inverse (which will be a smoothing operator) satisfies our assumptions; see section \ref{subsec: diff operator}. In fact, this case is even simpler.

The idea of formulating the equation as a convolution equation in order to study decay is taken from the classical paper \cite{Bona1997daa}, where the authors study the decay of solutions to equations of the form
\begin{equation*}
	u=K\ast G(u),
\end{equation*}
under a mild assumption on $G$ (see assumption (A3) below) and for $\widehat{K}\in H^s(\R)$ for some $s\geq 0$. Philosophically the idea is natural: the decay rate of $K$ should decide the decay rate of $u$, and if one can prove that, the problem is reduced to investigating the kernel $K$. In \cite{Bona1997daa} they show, under some integrability assumptions on $K$, that a solution $u$ that tends to $0$ at infinity decays at least as fast as $K$, and we will show that it will not decay faster. From Fourier analysis it is known that a requirement for $K$ to be exponentially decaying is that $\F(K)$ is analytical in a strip in the complex plane. Hence one would expect that if the symbol $m$ of $L$ is not smooth, solitary waves will decay only algebraically. This has been observed for instance for the Benjamin-Ono equation \cite{Benjamin1967iwo}, for which there is only one solitary wave and that one decays algebraically \cite{Amick1991uar}, and also for generalized KP equations \cite{Bouard1997sad}, both of which have Fourier symbols of finite smoothness. A more general result about the relation between finite smoothness and algebraic decay can be found in \cite{Cappiello2016sde}. We will assume smoothness of the symbol $m$ in this paper. To be precise, we will study \eqref{eq: main eq general} under the following assumptions:

\subsection*{Assumptions}
\begin{itemize}
	\item[(A1)] There is an $m_0<0$ such that
	\begin{equation*}
	|m^{(n)}(\xi)|\leq C_n (1+|\xi|)^{m_0-n}, \,\, n\in \N_0.
	\end{equation*}
	\item[(A1*)] The constants $C_n\geq 0$ in (A1) can be chosen such that $\lim_{n\rightarrow \infty} \frac{C_{n+1}/(n+1)!}{C_n/n!}=k$ for some $k\geq 0$.
	\item[(A2)] The function $m$ is even and the parameter $c$ satisfies
	\begin{equation*}
	\max_{\xi\in \R} m(\xi)<c.
	\end{equation*}
	\item[(A3)] $G\colon \R \rightarrow \R$ is bounded on compact sets, and for all small values of $u$, we have that $|G(u)|\lesssim |u|^r$ for some $r>1$.
\end{itemize}

\begin{remark}
	The assumptions (A1) and (A2) imply that $H_c$ decays algebraically of arbitrary order (cf. Section \ref{subsec: algebraic H_c}), while assumption (A1*) is needed for exponential decay. Indeed, (A1) implies that $m$ is real analytic and this has a local extension to a complex analytic function around every point in $\R$, and assumption (A1*) implies that there is a uniform lower bound on the radius for which the local extension is valid. Paley-Wiener theory can then be used to show exponential decay - see Section \ref{subsec: exponential H_c}. 
\end{remark}
Under these assumptions we have the following result on decay:

\begin{theorem}
	\label{thm: decay}
Let (A1), (A2) and (A3) be satisfied and suppose that $u\in L^\infty(\R)$ with $\lim_{|x|\rightarrow \infty} u(x)=0$ is a non-trivial solution to \eqref{eq: main eq general}. Then the following holds:
\begin{itemize}
	\item[(i)] $|\cdot|^l u(\cdot)\in L^\infty(\R)$ for any $l\geq 0$.
	\item[(ii)] If $m$ satisfies (A1*) in addition, then there is a number $\delta_c>0$ and an integer $n\geq 0$, depending on $m$ and $c$ (see Lemma \ref{lem: exponential decay kernel L2}), such that
	\begin{equation*}
	\mathrm{e}^{\delta_c|\cdot|}|u(\cdot)|
	\end{equation*}
	has algebraic growth of order $n$. That is, for all $\delta\in (0,\delta_c)$,
	\begin{equation*}
	\mathrm{e}^{\delta|\cdot|}u(\cdot) \in L^1\cap L^\infty (\R).
	\end{equation*}
	Moreover, $\mathrm{e}^{\delta_c|\cdot|}u(\cdot)\not \in L^p(\R)$ for any $p\in [1,\infty)$ and any $n\geq 0$, and $\mathrm{e}^{\delta_c|\cdot|}u(\cdot)\in L^\infty(\R)$ if and only if $n=0$.
\end{itemize}
\end{theorem}
It is also worth noting that while our inspiration comes from equations and systems for which solitary-wave solutions are solutions to equations of the form \eqref{eq: main} and we therefore work with \eqref{eq: main eq general}, the results apply to more general equations. Indeed, it is straightforward to extend the arguments to equations which can be cast in the form
\begin{equation*}
	u F(u)=H_c\ast G(u),
\end{equation*}
as long as $F\colon \R \rightarrow \R$ satisfies $\lim_{x\rightarrow 0} F(x)\neq 0$ and is such that Lemma \ref{lem: bound on u convolution} holds.

Under an assumption on $H_c$ that is independent of (A1) and (A2), and some assumptions on the behaviour of $G$ on the range of the solution, we have that decaying solutions to \eqref{eq: main eq general} are symmetric:
\begin{theorem}
	\label{thm: symmetry}
	Assume that $H_c\in L^1(\R)$ is non-negative, symmetric and monotonically decreasing on $(0,\infty)$, and that $G$ satisfies (A3). Let $u\in BC(\R)$ with $\lim_{|x|\rightarrow\infty} u(x)=0$ be a solution to \eqref{eq: main eq general} and assume that $G$ is non-negative, increasing, and satisfies $|G(x)-G(y)|\leq \tilde{c}|x-y|$ on the range of $u$, for some $0<\tilde{c}<c$. Then $u$ is symmetric about some point $\lambda_0\in \R$ and has exactly one crest, located at $\lambda_0$. 
\end{theorem}
\begin{remark}
\label{rem: assumptions}
	Some remarks on the assumptions:
	\begin{itemize}
		\item Note that we are requiring $G$ to be non-negative, increasing and Lipschitz continuous with Lipschitz constant $\tilde{c}<c$ only on the range of $u$, so these are implicitly assumptions on the solution $u$ itself.
		\item If $\frac{|G(x)-G(y)|}{|x-y|}\leq \left| \frac{G(x)}{x}+\frac{G(y)}{y}\right|$, then it is not necessary to assume that $|G(x)-G(y)|\leq \tilde{c}|x-y|$, as it follows from Lemma \ref{lem: touching}. This is the case if, for example, $G(u)=|u|^r$ for $1<r\leq 2$.
		\item If $f(\xi)=g(\xi^2)$ where $\lim_{x\rightarrow 0^+}g(x)<\infty$ and $\lim_{x\rightarrow \infty}g(x)=0$ and $g$ is completely monotone, then $\F^{-1}(f)$ is smooth outside the origin and monotone (Proposition 2.18 in \cite{Ehrnstrom2016owc}). As one can verify, if $m(\sqrt{\cdot})$ is completely monotone on $(0,\infty)$, then so is $\frac{m(\sqrt{\cdot})}{c-m(\sqrt{\cdot})}$. It follows that $m(0)>0$ and $m(\sqrt{\cdot})$ completely monotone on $(0,\infty)$ is sufficient for $H_c$ to be symmetric and monotone on $(0,\infty)$.
	\end{itemize}
\end{remark}

The paper is organized as follows. Section \ref{sec: H_c} is devoted to establishing integrability properties and the decay rate of $H_c$ under assumptions (A1), (A2) (and (A1*)). An exact description of the asymptotic behaviour, depending on $c$ and $m$, is given. In section \ref{sec: decay} we prove Theorem \ref{thm: decay}. Part (i) is more or less a straightforward adaption of the proof of algebraic decay of solitary waves for the Whitham equation in \cite{Bruell2017sad} (see also \cite{Bona1997daa}) and we do only part of the proof to show that the arguments of the aforementioned paper can indeed be applied. The proof of part (ii) is also an adaption of the arguments in \cite{Bona1997daa} and \cite{Bruell2017sad}, but we are able to give the exact rate of exponential decay. Moreover, in subsection \ref{subsec: diff operator}, the simpler case when $L$ is a differentiating operator, rather than smoothing as implied by assumption (A1), is discussed. In section \ref{sec: symmetry} symmetry is discussed and Theorem \ref{thm: symmetry} is proved. Finally, in section \ref{sec: examples}, the general results from the preceding sections are applied to some specific examples, in particular to the Whitham equation, and the bi-directional Whitham equation, giving the exact rate of exponential decay of solitary-wave solutions to these equations. This is an improvement on the results of \cite{Bruell2017sad}, where exponential decay of solitary waves of the Whitham equation is proved, but the exact decay rate is not established.

\section{Notation}
As indicated by the very definition of $L$ and $H_c$, we will make much use of the Fourier transform, for which we will use the normalization
\begin{equation*}
	\F(\varphi)(\xi)=\widehat{\varphi}(\xi)=\frac{1}{\sqrt{2\pi}}\int_\R \varphi(x)\mathrm{e}^{-ix\xi}\dx.
\end{equation*}
The inverse Fourier transform of $\varphi$ will be denoted by $\F^{-1}$ or $\check{\varphi}$ and is defined as
\begin{equation*}
	\check{\varphi}(x)=\frac{1}{\sqrt{2\pi}}\int_\R \varphi(\xi)\mathrm{e}^{ix\xi}\dxi.
\end{equation*}
With this normalization, the Fourier transform is a unitary operator on $L^2(\R)$.

For $s\geq 0$, the Sobolev space $H^s(\R)$ is the space of all $L^2(\R)$ functions $f$ which satisfy
\begin{equation*}
	\|f\|_{H^s(\R)}=\left(\int_\R (1+|\xi|^2)^s|\widehat{f}(\xi)|^2\dxi\right)^{1/2}<\infty.
\end{equation*}
The definition can be extended to $s<0$ by considering tempered distributions, but that is not relevant here.

\section{The kernel $H_c$}
\label{sec: H_c}
In this section we establish some essential properties of $H_c$, in particular its decay rate. We start by establishing integrability and algebraic decay; as one could expect assumption (A1*) is not necessary for these properties, only (A1) and (A2).

\subsection{Algebraic decay}
\label{subsec: algebraic H_c}
\begin{lemma}
\label{lem: Lp}
Assume (A1) and (A2) are satisfied. Then, for all $l\geq 1$, we have that
\begin{equation}
\label{eq: Lp}
	(\cdot)^l H_c(\cdot)\in L^p(\R), \quad \text{for all} \quad 2\leq p\leq \infty.
\end{equation}
\end{lemma}
\begin{proof}
Assumptions (A1) and (A2) imply that
\begin{equation*}
	\widehat{H_c}^{(j)}\in L^p(\R), \quad \text{for all} \quad 1\leq p\leq \infty, \,\, j\in \N_+.
\end{equation*}
As $\widehat{H_c}^{(j)}=\widehat{(-i \cdot)^{j}H_c}$ and $\F \colon L^p(\R) \rightarrow L^q(\R)$ for $1\leq p\leq 2$ and $\frac{1}{p}+\frac{1}{q}=1$, the result follows.
\end{proof}

\begin{lemma}
	\label{lem: behaviour at 0}
	Assume (A1) and (A2) are satisfied. Then, for $|x|\ll 1$, we have that
	\begin{equation*}
		|H_c(x)|\simeq
		\begin{cases}
		|x|^{-1-m_0}, \quad & -1<m_0<0, \\
		|\ln(|x|)|, \quad & m_0=-1,  \\
		1, \quad & m_0<-1.
		\end{cases}
	\end{equation*}
	That is $H_c\in L^\infty(\R)$ when $m_0<-1$.
\end{lemma}

\begin{proof}
	If $m_0<-1$, then $m\in L^1(\R)$ and by (A2) so is $\frac{m}{c-m}$ and the result is clear. Assume therefore that $-1<m_0<0$. Let
	\begin{equation*}
		g=\left(\frac{m}{c-m}\right)'=\frac{c m'}{(c-m)^2}.
	\end{equation*}
	As $m$ is even, we have that $g$ is odd and for $x>0$ (it is sufficient to consider $x>0$ as $m$, and therefore $H_c$, is even),
	\begin{align*}
		xH_c(x)= & \mathrm{i}\F^{-1}(g)(x)\\
		= & -\frac{1}{\sqrt{2\pi}}\int_\R g(\xi)\sin(x\xi)\dxi \\
		= & -\frac{1}{\sqrt{2\pi}}\int_\R g\left(\frac{s}{x}\right)\frac{\sin(s)}{x}\ds \\
		= & -\frac{2c}{\sqrt{2\pi}}\int_0^\infty \frac{\sin(s)}{x}\frac{1}{(c-m(s/x))^2}m'\left(\frac{s}{x}\right)\ds.
	\end{align*}
	By assumption (A1), we have that $|m'\left(\frac{s}{x}\right)|\lesssim \left(\frac{s}{x}\right)^{m_0-1}$. Moreover, $(c-m)^{-2}$ is bounded by assumption (A2). Hence
	\begin{equation*}
		x|H_c(x)|\lesssim x^{-m_0}\int_0^\infty \frac{|\sin(s)|}{s^{1-m_0}}\ds =C x^{-m_0}.
	\end{equation*}
	Dividing by $x$ on both sides gives the desired result. Now if $m_0=-1$, we use the estimate $|m'\left(\frac{s}{x}\right)|\lesssim \left(\frac{s}{x}\right)^{-2}$ for $s\geq x$ and the estimate $|m'\left(\frac{s}{x}\right)|\lesssim \left(\frac{s}{x}\right)^{-1}$ for $0<s<x$. Hence
	\begin{align*}
		x|H_c(x)|\lesssim & \int_0^x \frac{|\sin(s)|}{s}\ds+ x\int_x^\infty \frac{|\sin(s)|}{s^{2}}\ds \\
		\simeq & x+x|\ln(x)|,
	\end{align*}
	and the conclusion follows.
\end{proof}
From Lemmas \ref{lem: Lp} and \ref{lem: behaviour at 0}, we have the following Corollary:
\begin{corollary}
	\label{cor: integrability}
	Assume (A1) and (A2) are satisfied. Then $x\mapsto |x|^\alpha H_c(x)\in L^p(\R)$, $1\leq p\leq \infty$, if $\alpha> \max\lbrace 1+m_0,0\rbrace-\frac{1}{p}$.
\end{corollary}

\begin{proof}
	By \eqref{eq: Lp}, $|\cdot|^\alpha H_c(\cdot)\in L^p(\R\setminus (-1,1))$ for all $1\leq p\leq \infty$ and all $\alpha\in \R$. It remains to consider the behaviour around the origin. If $-1<m_0<0$, then by Lemma \ref{lem: behaviour at 0}, we have that $|x|^\alpha |H_c(x)|\simeq |x|^{-1-m_0+\alpha}$ which is in $L_{loc}^p(\R)$ exactly when $\alpha>1+m_0-\frac{1}{p}$. For $m_0\leq -1$, $|x|^\alpha |H_c(x)|\lesssim |x|^\alpha |\ln(|x|)|$ which is in $L_{loc}^p(\R)$ if and only if $\alpha>-\frac{1}{p}$.
\end{proof}

The results above state that $H_c$ decays algebraically with arbitrary order, is a bounded function away from the origin, with the behaviour at the origin being given by Lemma \ref{lem: behaviour at 0}. As $H_c$ decays faster than any polynomial, the natural question to ask is whether it decays exponentially. This is indeed the case, if assumption (A1*) is satisfied in addition to (A1) and (A2).

\subsection{Exponential decay}
\label{subsec: exponential H_c}
Now we turn the exponential decay of $H_c$, under the additional assumption (A1*). By assumption (A1) we have that $m$ is real analytic, and the lemma below shows that, under the additional assumption (A1*), $m$ can be extended to complex analytic function in a horizontal strip in the complex plane, centred at the real line.

\begin{lemma}
	\label{lem: exponential decay kernel L2}
	Let (A1), (A2) and (A1*) be satisfied. Considering $m$ as a function in the complex plane, let $\delta_c>0$ be the smallest number for which there exists a $z_0\in \C$ with $\mathrm{Im}\, z_0=\delta_c$ such that $m(z_0)=c$. Then the kernel $H_c$ can be expressed as
	\begin{equation*}
	H_c(x)=\mathrm{e}^{-\delta_c|x|}\left(v+P_n(|x|)\right), \,\, x\in \R,
	\end{equation*}
	where $v\in L^p(\lbrace x\in \R : |x|\geq 1\rbrace)$, $1\leq p\leq \infty$ satisfies Lemma \ref{lem: behaviour at 0} and $P_n$ is a polynomial of order $n$, where $n$ is the order of the zero of $m'$ at $z_0$. In particular, if $m'(z_0)\neq 0$, then $P_n$ is a non-zero constant.
\end{lemma}

\begin{proof}
	We want to use Paley-Wiener theory to show exponential decay. While $m$ is not guaranteed to be in $L^2$, assumption (A1) assures that $m'$ is, and we have
	\begin{equation*}
		\F (i\cdot H_c(\cdot))=\left(\frac{m}{c-m}\right)'=\frac{cm'}{(c-m)^2}:=g.
	\end{equation*}
	We claim that $g$ can be extended to a meromorphic function on a strip in the complex plane, potentially with branch cuts for $m$, centered around the real axis, containing at least one point where $m(z)=c$, and these are the only poles of $g$ in the strip. As $m$ is an analytic function on $\R$, we have at each point $x_0\in \R$ a local extension to the complex plane given by
	\begin{equation*}
	m(z)=\sum_{n=0}^{\infty} \frac{m^{(n)}(x_0)}{n!}(z-x_0)^n,
	\end{equation*}
	which is valid for all $z$ within a ball around $x_0$ with non-zero radius depending on $x_0$. Let
	\begin{equation*}
	\sigma := \inf_{x_0\in \R} \sup \lbrace r : \sum_{n=0}^{\infty} \frac{|m^{(n)}(x_0)|}{n!}r^n <\infty \rbrace.
	\end{equation*}
	That is, $\sigma$ is the infimum of the convergence radius over all points in $\R$. By (A1*), for any $x_0\in \R$ the convergence radius of the series above is greater than or equal to the convergence radius of
	\begin{equation*}
	\sum_{n=0}^{\infty} \frac{C_n(1+|x_0|)^{m_0-n}}{n!}z^n,
	\end{equation*}
	which converges for all $|z|<\frac{1}{k(1+|x_0|)^{m_0}}$ where $k$ is as in (A1*). Hence there is a lower bound on the convergence radius that is independent of the point $x_0$, and $\sigma\geq \frac{1}{k}>0$. Moreover, the convergence radius goes to infinity as $|x_0|\rightarrow \infty$. If $\delta_c<\sigma$, that is, there is a point $z_0$ such that $m(z_0)=c$ in this strip, we are done. Assume therefore that $\delta_c\geq\sigma$. The radius of convergence for the power series of $m$ is the distance to the closest singularity, which is either an isolated singularity or a branch point. In the former case, $m/(c-m)$ will be bounded around the singularity, and $g$ has a removable singularity at this point and can be extended further. The latter case of a branch cut for $m$ has no impact for our purposes. To see this, let $h(z)$ be real valued and even on the real line. By symmetry, we can without loss of generality assume that $h$ has a branch cut on the imaginary axis. As $h$ is even on the real line and holomorphic away from the branch cut, we have
	\begin{equation*}
		h(-\bar{z})=h(\bar{z})=\overline{h(z)}.
	\end{equation*}
	Hence $h$ has equal real part but opposite imaginary part on the two branches. Moreover, as $h$ is even $\mathrm{Re} \, h'(z)=0$ along the imaginary axis and
	\begin{equation*}
		h'(-\bar{z})=-h'(\bar{z})=-\overline{h'(z)}.
	\end{equation*}
	This implies that $h'$ has the same value on each part of the branch, hence it is not a necessary branch cut for $h'$. In general, if there are two or more branch cuts that are mirrored over the imaginary axis, the integrals of $g$ around them will cancel each other out, and we can without loss of generality assume that there are no branch cuts. This proves our original claim.
	
	Note that if $m(z_0)=c\in \R$, the symmetry of $m$ on the real line and the analyticity of $m$ around $z_0$ implies that
	\begin{equation*}
		m(-z_0)=m(z_0)=c=\overline{m(z_0)}=m(\overline{z_0})=m(-\overline{z_0}).
	\end{equation*}
	That is, the poles of $g$ are symmetric with respect to the real axis and the imaginary axis. As $g$ is the derivative of a function, we immediately get that the residue of $g$ at the poles $z_0=\xi \pm i\delta_c$ such that $m(z_0)=c$ is zero: for any $0<r<s$ such that $g(z)$ is analytic for $0<|z-z_0|<s$,
	\begin{equation*}
		\frac{1}{2\pi i}\int_{|z-z_0|=r} g(z) \dz =\frac{1}{2\pi i}\int_{|z-z_0|=r} \frac{\mathrm{d}}{\mathrm{d}z}\frac{1}{c-m(z)} \dz=0,
	\end{equation*}
	as the integral is taken over a closed circle. Let $n$ be the order of the zero of $m'$ at $z_0$. Then, locally around $z_0$,
	\begin{equation*}
		g(z)=c\frac{(n+1)!(n+1)}{m^{(n+1)}(z_0)}(z-z_0)^{-n-2}+O((z-z_0)^{-n-1}).
	\end{equation*}
	That is, $g$ and therefore also $g(z)\mathrm{e}^{ixz}$, $x\in \R\setminus\lbrace 0\rbrace$, has a pole of order $n+2$ at $z_0$.  It follows that
	\begin{align}
		\mathrm{Residue}\, [g(z)\mathrm{e}^{ixz}, z_0]= & \frac{1}{(n+1)!}\lim_{z\rightarrow z_0} \frac{\mathrm{d}^{n+1}}{\mathrm{d} z^{n+1}}[g(z)\mathrm{e}^{ixz}(z-z_0)^{n+2}] \nonumber \\
		= & \sum_{k=0}^{n+1} \frac{1}{k!(n+1-k)!}\left[\lim_{z\rightarrow z_0}\frac{\mathrm{d}^{n+1-k}}{\mathrm{d} z^{n+1-k}}g(z)(z-z_0)^{n+2}\right]i^k x^k \mathrm{e}^{ixz_0} \nonumber \\
		=: & \mathrm{e}^{ixz_0} Q_{n+1}(x) \label{eq: residue at z_0 1},
	\end{align}
	where $Q_{n+1}(x)$ is a polynomial of order $n+1$. As the residue of $g$ is zero, $Q_{n+1}(0)=0$, and we can write $Q_{n+1}$ as
	\begin{equation}
		\label{eq: Q}
		Q_{n+1}(x)=\sum_{k=1}^{n+1} a_k x^k, \quad a_{n+1}=i^{n+1} c\frac{n+1}{m^{(n+1)}(z_0)}\neq 0.
	\end{equation}
	In particular, $\mathrm{Residue}\, [g(z)\mathrm{e}^{ixz}, z_0]\neq 0$ for $x\neq 0$. Furthermore, as $g(-\overline{z})=-\overline{g(z)}$, we see that the residue of $g(z)\mathrm{e}^{ixz}$ at $z_0$ is purely real if $z_0=\pm i\delta_c$. If $z_0=\xi_0 \pm i \delta_c$, $\xi_0\neq 0$, then imaginary parts of the residues at $z_0$ and $-\overline{z_0}$ cancel each other, while the real parts add up. For simplicity we therefore assume that we only have poles at $\pm i\delta_c$. From \eqref{eq: residue at z_0 1} we then get
	\begin{equation}
	\label{eq: residue at z_0}
	\mathrm{Residue}\, [g(z)\mathrm{e}^{ixz}, \pm i\delta_c]=\mathrm{e}^{\mp x\delta_c}Q_{n+1}(x)
	\end{equation}
	for any $x\in \R$.
	
	Let $y>0$ be fixed and $|x|>|y|k$. By (A1*),
	\begin{align}
	|m'(x+iy)|\leq & \sum_{n=0}^\infty \left|\frac{m^{(n+1)}(x)}{n!} (iy)^n\right| \nonumber\\
	\leq & \sum_{n=0}^\infty \frac{C_{n+1}}{n!}|y|^n (1+|x|)^{m_0-n-1} \nonumber \\
	= & (1+|x|)^{m_0-1} \sum_{n=0}^\infty \frac{C_{n+1}}{n!} \left(\frac{|y|}{1+|x|}\right)^n \nonumber \\
	\leq & K(1+|x|)^{m_0-1}, \label{eq: decay of m'(x+iy)}
	\end{align}
	for some $K>0$ (that will in fact decrease as $|x|$ increases). As $m_0<0$, this is in $L^p(\R)$ for all $1\leq p\leq \infty$. This also implies that for $x>0$,
	\begin{equation}
	\label{eq: vertical sides x>0}
	\lim_{|\xi| \rightarrow \infty} \sup_{0\leq \eta \leq \delta_c} |g(\xi+i\eta)\mathrm{e}^{ixz}|=0,
	\end{equation}
	as $\mathrm{e}^{ixz}$ is bounded for $x>0$ and $\mathrm{Im}\, z\geq 0$, where $z=\xi+i\eta$.
	If $x<0$ and $\mathrm{Im}\, z\leq 0$, then $\mathrm{e}^{ixz}$ is also bounded and
	\begin{equation*}
	\lim_{|\xi| \rightarrow \infty} \sup_{-\delta_c \leq \eta \leq 0} |g(\xi+i\eta)\mathrm{e}^{ixz}|=0.
	\end{equation*}
	We consider first the case $x>0$; the case when $x<0$ is similar. The function $g(\cdot +i\delta_c)$ is not in $L^2(\R)$, but for every $\varepsilon>0$, we have $g(\cdot+i\delta_c)\in L^2(\R\setminus [-\varepsilon,\varepsilon])$ by \eqref{eq: decay of m'(x+iy)}. To apply Cauchy's theorem, we consider the domain defined by the line segments $[-R,R]$, $[\pm R, \pm R+i\delta_c]$, $[-R+i\delta_c, -\varepsilon+i\delta_c]$, and $[\varepsilon+i\delta_c, R+i\delta_c]$ and the half-circle $\Gamma_\varepsilon=\lbrace z=i\delta_c+\varepsilon\mathrm{e}^{i\theta} : \pi\leq \theta \leq 2\pi\rbrace$. By \eqref{eq: vertical sides x>0}, the integral over the vertical lines vanish as $R\rightarrow \infty$. As $g(z)\mathrm{e}^{ixz}$ has no singularities in this domain, Cauchy's theorem then gives (letting $R\rightarrow \infty$)
	\begin{equation}
	\label{eq: indented rectangle cauchy}
	\frac{1}{\sqrt{2\pi}}\int_\R g(\xi)\mathrm{e}^{ix\xi}\dxi=\mathrm{e}^{-\delta_c x}\frac{1}{\sqrt{2\pi}}\int_{|\xi|\geq \varepsilon} g(\xi+i\delta_c)\mathrm{e}^{ix\xi}\dxi+ \frac{1}{\sqrt{2\pi}}\int_{\Gamma_\varepsilon} g(z)\mathrm{e}^{ixz}\dz.
	\end{equation}
	Recall that $(2\pi)^{-1/2}\int_\R g(\xi)\mathrm{e}^{ix\xi}\dxi=-ix H_c(x)$. Multiplying by $\mathrm{e}^{\delta_c x}$ on both sides, we get
	\begin{equation*}
	-\mathrm{e}^{x\delta_c}ixH_c(x)=\frac{1}{\sqrt{2\pi}}\int_{|\xi|\geq \varepsilon} g(\xi+i\delta_c)\mathrm{e}^{ix\xi}\dxi+\frac{1}{\sqrt{2\pi}} \mathrm{e}^{\delta_c x}\int_{\Gamma_\varepsilon} g(z)\mathrm{e}^{ixz}\dz.
	\end{equation*}
	The first term on the right hand side is in $L^2(\R)$. For the second term, the (fractional) residue theorem gives that for $\varepsilon>0$ small enough,
	\begin{align*}
	\mathrm{e}^{\delta_c x}\frac{1}{\sqrt{2\pi}}\int_{\Gamma_\varepsilon} g(z)\mathrm{e}^{ixz}\dz=& \mathrm{e}^{\delta_c x} \frac{1}{\sqrt{2\pi}} \left(i\pi \mathrm{Residue}\, [g(z)\mathrm{e}^{ixz},i\delta_c]+O(\varepsilon)\right)\\
	=& \sqrt{\frac{\pi}{2}} iQ_{n+1}(x)+O(\varepsilon).
	\end{align*}
	This is clearly not in $L^p(\R)$ for any $p\in[1,\infty]$. These calculations were for $x>0$; if $x<0$ we consider the conjugate of the indented rectangle and we obtain the equivalent result. To get the expression for $H_c$, consider the function $\xi^{n+2}g(\xi+i\delta_c)$. Taylor expanding around $\xi=0$, we get by \eqref{eq: residue at z_0 1} and \eqref{eq: Q} that
	\begin{equation*}
	\xi^{n+2}g(\xi+i\delta_c)=\sum_{k=0}^n \frac{(n+1-k)!}{(n+1)!} i^{k-n-1} a_{n+1-k}\xi^{k} +O(\xi^{n+2}),
	\end{equation*}
	and hence
	\begin{equation*}
	g(\xi+i\delta_c)=\sum_{k=0}^n (n+1-k)! i^{k-n-1} a_{n+1-k}\xi^{k-n-2}+O(1)
	\end{equation*}
	for small $\xi$. As $g(\cdot+i\delta_c)\in L^p(\R\setminus (-1,1))$ for all $1\leq p\leq \infty$, we have that
	\begin{equation*}
	g(\xi+i\delta_c)=\sum_{k=0}^n (n+1-k)! i^{k-n-1} a_{n+1-k}\xi^{k-n-2}+w(\xi),
	\end{equation*}
	where $w\in L^p(\R)$ for all $1\leq p\leq \infty$. Hence the "ill-behaved" part of $\int_{|\xi|\geq \varepsilon} g(\xi+i\delta_c)\mathrm{e}^{ix\xi}\dxi$ can be explicitly calculated as $\varepsilon\rightarrow 0^+$, as the limit is symmetric (otherwise it is not defined). The calculation is straightforward calculus and the result can be found in any table of Fourier transforms, and we find that (again we are assuming $x>0$)
	\begin{equation*}
	\lim_{\varepsilon\rightarrow 0^+} \frac{1}{\sqrt{2\pi}}\int_{|\xi|\geq \varepsilon} g(\xi+i\delta_c)\mathrm{e}^{ix\xi}\dxi= \sqrt{\frac{\pi}{2}} iQ_{n+1}(x) + \check{w}(x),
	\end{equation*}
	where $\check{w}\in L^p(\R)$ for all $2\leq p\leq \infty$. Hence, taking the limit $\varepsilon\rightarrow 0^+$ in \eqref{eq: indented rectangle cauchy}, we get
	\begin{equation}
	\label{eq: expression for H_c}
	-ixH_c(x)=\mathrm{e}^{-\delta_c|x|}\left(\check{w}(x)-i\sqrt{2\pi} Q_{n+1}(x)\right).
	\end{equation}
	Dividing by $-ix$ for $x\neq 0$ gives the expression for $H_c$, with $v=-\frac{\check{w}}{ix}$ and $P_n(x)=\sqrt{2\pi}Q_{n+1}(x)x^{-1}$ (recall that $Q_{n+1}(0)=0$, so that $Q_{n+1}(x)x^{-1}$ is indeed a polynomial).
	 
\end{proof}

\begin{corollary}
	\label{cor: Lp kernel}
	Let (A1), (A2) and (A1*) be satisfied, and let $\delta_c$ be as in Lemma \ref{lem: exponential decay kernel L2}. Then, for all $0<\delta<\delta_c$, we have that $\mathrm{e}^{\delta|\cdot|}H_c(\cdot)\in L^p(\R)$ for all $1\leq p<\frac{1}{1+m_0}$ if $-1< m_0<0$, all $1\leq p<\infty$ if $m_0=-1$, and all $1\leq p\leq \infty$ if $m_0<-1$.
\end{corollary}

\begin{proof}
	Let $\delta$ as in the assumptions. By Lemma \ref{lem: exponential decay kernel L2}, we have that
	\begin{equation*}
	\mathrm{e}^{\delta|x|}H_c(x)=\mathrm{e}^{-(\delta_c-\delta)|x|}\left(v(x)+P_n(|x|)\right).
	\end{equation*}
	As $v\in L^p(\lbrace x\in \R : |x|\geq 1\rbrace)$ for all $1\leq p\leq \infty$ and $\delta_c-\delta>0$, we get that $\mathrm{e}^{\delta|\cdot|}H_c(\cdot)\in L^p(\lbrace x\in \R : |x|\geq 1\rbrace)$, hence we need only check the behaviour at $0$. If $m_0<-1$, then by Lemma \ref{lem: behaviour at 0}, we have that $v\in L^\infty (\R)$, and the conclusion follows. For $-1\leq m_0<0$, the result follows from Corollary \ref{cor: integrability} with $\alpha=0$.
\end{proof}

We have established the precise decay rate of $H_c$, which is sufficient to establish the precise decay rate of (decaying) solutions to \eqref{eq: main eq general} (see Section \ref{sec: decay} below).

\section{Decay of solitary waves}
\label{sec: decay}
With the properties of $H_c$ established in Section \ref{sec: H_c}, we can now establish the decay properties of solutions to \eqref{eq: main eq general}, under assumption (A3) on $G$. We start with algebraic decay.

\subsection{Algebraic decay of solitary waves}
\begin{theorem}
	\label{thm: algebraic decay}
	Let (A1), (A2) and (A3) be satisfied and suppose that $u\in L^\infty(\R)$ with $\lim_{|x|\rightarrow \infty} u(x)=0$ is a solution to \eqref{eq: main eq general}. Then
	\begin{equation*}
		(\cdot)^l u(\cdot)\in L^q(\R)
	\end{equation*}
	for all $l\geq 0$ and all $q\in (2,\infty)$.
\end{theorem}

\begin{proof}
	Choose $p\in (1,2)$ and let $\alpha=\alpha(p)$ a constant satisfying
	\begin{equation*}
		\alpha>1-\frac{1}{p}.
	\end{equation*}
	In particular, $\alpha$ satisfies the condition in Corollary \ref{cor: integrability}, so that $(1+|\cdot|)^\alpha H_c(\cdot)\in L^p(\R)$. Let
	As $|G(u)|\lesssim |u|^r$ for some $r>1$ and $\lim_{|x|\rightarrow \infty} u(x)=0$, we get that for any $\delta>0$, there exists an $R_\delta \geq 0$ such that
	\begin{equation*}
		|G(u(x))|\leq \delta |u(x)| \,\, \text{for all} \,\, |x|\geq R_\delta.
	\end{equation*}
	Picking $0<\delta < c$, we get that
	\begin{equation*}
		c-\frac{G(u(x))}{u(x)}\geq c-\delta>0 \,\, \text{for all} \,\, |x|\geq R_\delta
	\end{equation*}
	As $u$ is a solution to \eqref{eq: main eq general}, we get that
	\begin{align*}
		\left|u(x)\left(c-\frac{G(u(x))}{u(x)}\right)\right| & = \left|\int_\R H_c(x-y)(1+|x-y|)^\alpha \frac{G(u(y))}{(1+|x-y|)^\alpha}\dy \right| \\
		& \leq \int_\R |H_c(x-y)|(1+|x-y|)^\alpha \frac{|G(u)|}{(1+|x-y|)^\alpha}\dy.
	\end{align*}
	Letting $q$ be the conjugate of $p$, we get by H\"older's inequality that
	\begin{equation}
		|u(x)|\leq C\left(\int_\R \frac{|G(u)|^{q}}{(1+|x-y|)^{\alpha q}}\dy \right)^{1/q} \,\, \text{for all} \,\, |x|\geq R_\delta,
	\end{equation}
	where $C=C_{\alpha, p, \delta}=(c-\delta)^{-1}\|(1+|\cdot|)^\alpha H_c(\cdot)\|_{L^p(\R)}<\infty$.
	The rest of the proof then follows that of Theorem 3.9 in \cite{Bruell2017sad}, with the obvious modifications.
\end{proof}
With this result it is simple to prove part (i) of Theorem \ref{thm: decay}:

\begin{proof}[Proof of Theorem \ref{thm: decay} (i)]
	 As shown in the proof of Theorem \ref{thm: algebraic decay}, for every $0<\delta<c$ there is a $R_\delta\geq 0$ such that
	\begin{equation*}
		c-\frac{G(u(x))}{u(x)}\geq c-\delta>0 \,\, \text{for all} \,\, |x|\geq R_\delta.
	\end{equation*}
	Pick one such $\delta$ and $R_\delta$. Since $u\in L^\infty(\R)$, we have that $|\cdot|^l u(\cdot)$ is bounded on bounded sets, so it remains only to consider $|x|\geq R_\delta$.
	From \eqref{eq: main eq general} and repeated use of H\"older's inequality, we get
	\begin{align*}
		|x|^l|u(x)|\leq & (c-\delta)^{-1}\int_\R |x-y|^l |H_c(x-y)||G(u(y))|\dy +\int_\R |H_c(x-y)| |y|^l |G(u(y))|\dy \\
		\lesssim & \||\cdot|^l H_c(\cdot)\|_{L^1(\R)}\|G(u)\|_{L^\infty(\R)}+\|u^{r-1}\|_{L^\infty(\R)}\int_\R |H_c(x-y)||y|^l|u(y)|\dy \\
		\lesssim & \||\cdot|^l H_c(\cdot)\|_{L^1(\R)}\|u\|_{L^\infty(\R)}^r+\|u\|_{L^\infty(\R)}^{r-1}\|H_c\|_{L^p(\R)}\||\cdot|^lu(\cdot)\|_{L^q(\R)},
	\end{align*}
	where $\frac{1}{p}+\frac{1}{q}=1$. By Corollary \ref{cor: integrability} the first term in the last line is bounded and we can find $p\in (1,2)$ such that $H_c\in L^p(\R)$. Then $q\in (2,\infty)$ and by Theorem \ref{thm: algebraic decay} the last term is also bounded. The constant implied in the notation $\lesssim$ can be taken independently of $x$, and the conclusion follows.
\end{proof}

\subsection{Exponential decay of solitary waves}
In this section we will add assumption (A1*).

\begin{lemma}
	\label{lem: bound on u convolution}
	Let (A1), (A2), (A3) and (A1*) be satisfied. Suppose that $u\in L^\infty(\R)$ with $\lim_{|x|\rightarrow \infty} u(x)=0$ is a solution to \eqref{eq: main eq general}. Then there exists a constant $C>0$ such that
	\begin{equation*}
		|u(x)|\leq C \int_\R |H_c(x-y)||G(u(y))|\dy 
	\end{equation*}
	for almost every $x\in \R$.
\end{lemma}

\begin{proof}
	From \eqref{eq: main eq general} we get that
	\begin{align*}
		|u(x)|\left|c-\frac{G(u(x))}{u(x)}\right|=\left|\int_\R H_c(x-y)G(u(y))\dy\right|\leq \int_\R |H_c(x-y)||G(u(y))|\dy.
	\end{align*}
	If $\left|c-\frac{G(u(x))}{u(x)}\right|\geq \gamma>0$, then
	\begin{equation*}
		|u(x)|\leq \gamma^{-1}\int_\R |H_c(x-y)||G(u(y))|\dy.
	\end{equation*}
	As shown in the proof of Theorem \ref{thm: algebraic decay}, for any $\gamma\in (0,c)$, there exists an $R_\gamma>0$ such that $\left|c-\frac{G(u(x))}{u(x)}\right|\geq \gamma>0$ for all $|x|\geq R_\gamma$. It follows that the set
\begin{equation*}
E_\gamma=\lbrace x\in \R : \left|c-\frac{G(u(x))}{u(x)}\right|< \gamma \rbrace,
\end{equation*}	
	  is contained in bounded interval, and moreover that $\inf_{x\in E_\gamma} |u(x)|\geq C>0$ for some $C>0$. We have that
	  \begin{equation*}
	  |cu(x)-G(u(x))|<\gamma |u(x)|\leq \gamma \|u\|_{L^\infty(\R)}, \,\, x\in E_\gamma.
	  \end{equation*}
	  By \eqref{eq: main}, $cu-G(u)=L(u)\in C(\R)$ (since $L$ is smoothing), and it follows that $G(u(x))$ is non-zero in some interval around $x$ for all $x\in E_\gamma$. As $H_c$ is non-zero around the origin and $E_\gamma$ is a subset of a compact set, it follows that
	\begin{equation*}
		I_\gamma:=\inf \lbrace \int_\R |H_c(x-y)||G(u(y))|\dy : x\in E_\gamma \rbrace >0.
	\end{equation*}
	Hence, for any $\gamma \in (0,c)$, we have that $\max \lbrace \gamma^{-1}, I_\gamma^{-1} \|u\|_{L^\infty(\R)}\rbrace <\infty $ and
	\begin{equation*}
		|u(x)|\leq \max \lbrace \gamma^{-1}, I_\gamma^{-1} \|u\|_{L^\infty(\R)} \rbrace \int_\R |H_c(x-y)||G(u(y))|\dy,
	\end{equation*}
	which guarantees the existence of a $C$ such as in the statement.
\end{proof}
Now we will prove our main result, part (ii) of Theorem \ref{thm: decay}.

\begin{proof}[Proof of Theorem \ref{thm: decay} (ii)]
	First we want to show that
	\begin{equation*}
	\mathrm{e}^{\delta|\cdot|}u(\cdot)\in L^1(\R)\cap L^\infty(\R) \,\, \text{for any}\,\, \delta\in [0,\delta_c).
	\end{equation*}
	The proof of this follows largely the arguments of Corollary 3.1.4 in \cite{Bona1997daa}, with some adaptations (see also Theorem 3.12 in \cite{Bruell2017sad}), but we include the details for completeness.
	If $-\frac{1}{2}\leq m_0<0$, let $1<p<\frac{1}{1+m_0}$; otherwise let $1<p<2$. Let $q$ be the H\"older conjugate of $p$ and let $\delta \in (0,\delta_c)$. Let $M_1$ be the smallest constant such that
	\begin{equation}
		\label{eq: M_1}
		|u(x)|\leq M_1 \int_\R |H_c(x-y)||G(u(y))|\dy \,\, \text{for all} \,\, x\in \R,
	\end{equation}
	and set
	\begin{align*}
		M_2 & =\|(\cdot)G(u)u^{-1}\|_{L^\infty(\R)} \\
		M_3 & =\|\mathrm{e}^{\delta|\cdot|}H_c(\cdot)\|_{L^p(\R)}.
	\end{align*}
	The boundedness of $M_1$ and $M_2$ follows from Lemma \ref{lem: bound on u convolution} and Theorem \ref{thm: decay} (i), respectively, and $M_3$ is bounded by Corollary \ref{cor: Lp kernel}. 
	Let
	\begin{equation*}
		D:=\max\lbrace 1, \frac{\delta}{2}\|u\|_{L^1(\R)}, M_1 M_2 M_3 \delta^{1/p}\left(\frac{2}{q}\right)^{1/q}\rbrace.
	\end{equation*}
	We claim that
	\begin{equation}
	\label{eq: claim}
		\|(\cdot)^l u(\cdot)\|_{L^1(\R)}\leq \frac{(l+2)!D^{l+1}}{\delta^{l+1}}, \,\, \text{for all} \,\, l\in \N.
	\end{equation}
	Clearly it is true for $l=0$. Assume it is true for $l=1,2, ... n$.
	Recall the following identity that can be proved by induction:
	\begin{equation*}
		x^n (f\ast g)(x)=\sum_{j=0}^{n}\binom{n}{j}\left((\cdot)^{n-j}f\ast (\cdot)^j g\right)(x).
	\end{equation*}
	Using this identity, Young's inequality and \eqref{eq: M_1}, we find that
	\begin{align*}
		\|(\cdot)^{n+1}u(\cdot)\|_{L^1(\R)}\leq & M_1\|(\cdot)^{n+1}(H_c\ast G(u))(\cdot)\|_{L^1(\R)} \\
		\leq & M_1\sum_{j=0}^{n+1}\binom{n+1}{j}\|(\cdot)^{n+1-j}H_c(\cdot)\|_{L^1(\R)}\|(\cdot)^{j}G(u(\cdot))\|_{L^1(\R)}.
	\end{align*}
	 Considering the term involving $H_c$ first, we get by H\"older's inequality:
	\begin{align*}
		\int_\R |x^{n+1-j}H_c(x)|\dx &\leq \int_\R |x^{n+1-j}\mathrm{e}^{-\delta|x|}||\mathrm{e}^{\delta|x|}H_c(x)|\dx \\
		& \leq \left(\int_\R |\mathrm{e}^{\delta|x|}H_c(x)|^p\dx\right)^{1/p}\left(\int_\R |x|^{q(n+1-j)} \mathrm{e}^{-q\delta|x|}\dx \right)^{1/q} \\
		& =M_3 2^{1/q}\left(\int_0^\infty x^{q(n+1-j)} \mathrm{e}^{-q\delta|x|}\dx \right)^{1/q} \\
		& = M_3 2^{1/q} \left(\frac{(q(n+1-j))!}{(q\delta)^{q(n+1-j)+1}}\right)^{1/q} \\
		& \leq M_3 \left(\frac{2}{q}\right)^{1/q}\frac{(n+1-j)!}{\delta^{n+1-j+1/q}}.
	\end{align*}
	And for the term involving $G$, we have that for $1\leq j\leq n+1$,
	\begin{align*}
		\|(\cdot)^j G(u(\cdot))\|_{L^1(\R)} & =\int_\R |x|^j \frac{G(u(x))}{u(x)} u(x)\dx \\
		& \leq M_2 \|(\cdot)^{j-1}u(\cdot)\|_{L^1(\R)} \\
		& \leq M_2 \frac{(j+1)!D^j}{\delta^j}.
	\end{align*}
	Thus we get that
	\begin{align*}
		\|(\cdot)^{n+1}u(\cdot)\|_{L^1(\R)} & \leq M_1 M_2 M_3 \left(\frac{2}{q}\right)^{1/q} \sum_{j=0}^{n+1}\binom{n+1}{j} \frac{(n+1-j)!(j+1)! D^j}{\delta^{n+1+1/q}} \\
		& =M_1 M_2 M_3 \delta^{1/p}\left(\frac{2}{q}\right)^{1/q} \sum_{j=0}^{n+1} \frac{(n+1)!(j+1) D^j}{\delta^{n+1+1/q+1/p}} \\
		& \leq \sum_{j=0}^{n+1} \frac{(n+1)! (j+1)D^{j+1}}{\delta^{n+2}} \\
		& = \frac{(n+3)!D^{n+2}}{\delta^{n+2}},
	\end{align*}
	which proves the claim. Applying \eqref{eq: claim},
	\begin{align*}
		\int_\R \mathrm{e}^{\nu|x|} |u(x)|\dx \leq & \sum_{l=0}^{\infty}\frac{\nu^l}{l!}\int_\R |x|^l|u(x)|\dx \\
		\leq & \sum_{l=0}^{\infty}\frac{\nu^l}{l!}\frac{(l+2)!D^{l+1}}{\delta^{l+1}} \\
		\leq & \sum_{l=0}^{\infty}\frac{\nu^l (l+2)(l+1)D^{l+1}}{\delta^{l+1}}.
	\end{align*}
	Hence the integral converges if $0<\nu<\frac{\delta}{D}$, and it follows that $\mathrm{e}^{\nu|\cdot|}u(\cdot)\in L^1(\R)\cap L^\infty(\R)$ for some $0<\nu<\delta$. Let $\eta=\sup \lbrace \nu : \mathrm{e}^{\nu|\cdot|}u(\cdot)\in L^1(\R)\cap L^\infty(\R)\rbrace$. Assume $\eta<\delta$, and choose $\nu$ such that
	\begin{equation*}
		\frac{\eta}{r}<\nu<\min\lbrace \eta, \frac{\delta}{r}\rbrace.
	\end{equation*}
	We have that
	\begin{align*}
		|u(x)|\mathrm{e}^{r\nu |x|}\leq & M_1 \int_\R |H_c(x-y)|\mathrm{e}^{r\nu|x-y|} |G(u)|\mathrm{e}^{r\nu|y|}\dy \\
		\lesssim & M_1 \int_\R |H_c(x-y)|\mathrm{e}^{2\nu|x-y|} \left(u(y)\mathrm{e}^{\nu|y|}\right)^r\dy \\
		= & M_1 \left(H_c(\cdot)\mathrm{e}^{r\nu|\cdot|}\ast \left(u(\cdot)\mathrm{e}^{\nu|\cdot|}\right)^r\right)(x).
	\end{align*}
	By Young's inequality, we get
	\begin{equation*}
		\|\mathrm{e}^{r\nu |\cdot|}u(\cdot)\|_{L^1(\R)}\leq M_1 \|H_c(\cdot)\mathrm{e}^{r\nu|\cdot|}\|_{L^1(\R)} \|\left(u(\cdot)\mathrm{e}^{\nu|\cdot|}\right)^r\|_{L^1(\R)}<\infty,
	\end{equation*}
	and
	\begin{equation*}
		\|\mathrm{e}^{r\nu |\cdot|}u(\cdot)\|_{L^\infty(\R)}\leq M_1 \|H_c(\cdot)\mathrm{e}^{r\nu|\cdot|}\|_{L^1(\R)} \|\left(u(\cdot)\mathrm{e}^{\nu|\cdot|}\right)^r\|_{L^\infty(\R)}<\infty.
	\end{equation*}
	But as $r>1$ we have that $r\nu>\eta$, and this contradicts the definition of $\eta$. Hence the assumption that $\eta<\delta$ must be false, and it must be the case that $\eta\geq \delta$. As $\delta\in (0,\delta_c)$ was arbitrary, this shows that 
	\begin{equation*}
	\mathrm{e}^{\delta|\cdot|}u(\cdot)\in L^1(\R)\cap L^\infty(\R) \,\, \text{for any}\,\, \delta\in [0,\delta_c).
	\end{equation*}
	Assume now that $n=0$ in \ref{lem: exponential decay kernel L2}, so that $H_c(x)=\mathrm{e}^{-\delta_c|x|}(v+C)$ for some constant $C\neq 0$, and let $f=v+C$; then $f\in L_{loc}^1(\R)$ and $f$ is bounded for $|x|>1$. We have that
	\begin{equation*}
		\mathrm{e}^{\delta_c|x|}|u(x)|\lesssim \int_\R |H_c(x-y)|\mathrm{e}^{\delta_c|x-y|} G(u(y))\mathrm{e}^{\delta_c |y|}\dy \simeq  \int_\R |f(x-y)|\left(u(y)\mathrm{e}^{\frac{\delta_c}{r}|y|}\right)^r\dy.
	\end{equation*}
	Splitting the integral into the integral over $|x-y|<1$, and $|x-y|\geq 1$ and applying H\"older's inequality, we get
	\begin{align*}
		\mathrm{e}^{\delta_c|x|}|u(x)| \lesssim & \|f\|_{L^1((-1,1))}\|u(\cdot)\mathrm{e}^{\frac{\delta_c}{r}|\cdot|}\|_{L^\infty(\R)}^r \\
		& + \|f\|_{L^\infty(\R\setminus(-1,1))} \|u(\cdot)\mathrm{e}^{\frac{\delta_c}{r}|\cdot|}\|_{L^\infty(\R)}\|u(\cdot)\mathrm{e}^{\frac{\delta_c}{r}|\cdot|}\|_{L^1(\R)}.
	\end{align*}
	The right hand side is finite and independent of $x$, hence we conclude that $\mathrm{e}^{\delta_c|\cdot|}u(\cdot)\in L^\infty(\R)$. Now we want to show that this is optimal. Let $\varepsilon>0$. By the decay of $u$ and assumption (A3), we have that
	\begin{equation*}
		|H_c(x-y)G(u(y))|\lesssim \mathrm{e}^{-\delta_c|x-y|}|f(x-y)|\mathrm{e}^{-r\delta_c|y|}\leq \mathrm{e}^{-\delta_c|x|}|f(x-y)|\mathrm{e}^{-(r-1)\delta_c|y|},
	\end{equation*}
	for all $|y|$ sufficiently large. As $r>1$ and $f\in L_{loc}^1(\R)$, we can find $R_\varepsilon$ such that
	\begin{equation*}
		\left|\int_{|y|>R_\varepsilon} H_c(x-y) G(u(y))\dy\right|<\varepsilon \mathrm{e}^{-\delta_c|x|}.
	\end{equation*}
	Now let $|x|>R_\varepsilon$ be such that $f(x-y)=C+O(\varepsilon)$ for all $|y|\leq R_\varepsilon$. This is possible as $\lim_{|x|\rightarrow \infty} f(x)=C\neq 0$. If $x>R_\varepsilon$, we get that
	\begin{align*}
		\mathrm{e}^{\delta_c|x|}u(x)  \simeq & \mathrm{e}^{\delta_c|x|} \int_\R \mathrm{e}^{-\delta_c|x-y|}f(x-y)G(u(y))\dy \\
		 = &\mathrm{e}^{\delta_c|x|} \int_{|y|\leq R_\varepsilon}\mathrm{e}^{-\delta_c|x-y|}f(x-y)G(u(y))\dy \\
		& +\mathrm{e}^{\delta_c|x|}\int_{|y|> R_\varepsilon}\mathrm{e}^{-\delta_c|x-y|}f(x-y)G(u(y))\dy \\
		= & C\int_{|y|\leq R_\varepsilon} \mathrm{e}^{\delta_c y}G(u(y))\dy +O(\varepsilon),
	\end{align*}
	and if $x<-R_\varepsilon$, we get
	\begin{equation*}
		\mathrm{e}^{\delta_c|x|}u(x)  \simeq C\int_{|y|\leq R_\varepsilon} \mathrm{e}^{-\delta_c y}G(u(y))\dy +O(\varepsilon)
	\end{equation*}
	As $G(u)$ is non-zero on a set of non-zero measure,
\begin{equation*}
\int_{|y|\leq R_\varepsilon} \mathrm{e}^{-\delta_c y}G(u(y))\dy \quad \text{and} \quad \int_{|y|\leq R_\varepsilon} \mathrm{e}^{\delta_c y}G(u(y))\dy
\end{equation*}	
cannot both converge to $0$ as $\varepsilon\rightarrow 0^+$. This shows that $\mathrm{e}^{\delta_c|x|}u(x)$ does not decay to $0$ as $|x|\rightarrow \infty$, and it also implies that $\mathrm{e}^{\delta_c|\cdot|}u(\cdot)\in L^p(\R)$ only for $p=\infty$. This was for $n=0$; by the same arguments we see that if $n>0$, then $\mathrm{e}^{\delta_c|\cdot|}|u(\cdot)|$ has algebraic growth of order $n$.
\end{proof}

\subsection{When $L$ is a differentiating operator}
\label{subsec: diff operator}
Assumption (A1) implies that $L$ is a smoothing operator, and the dispersion in \eqref{eq: main} is very weak. However, our results can easily be extended to the case with stronger dispersion as well, by making a few observations. As shown in the introduction, \eqref{eq: main} can formally be written as
\begin{equation*}
u=\F^{-1}\left(\frac{1}{c-m}\right)\ast G(u).
\end{equation*}
If $m>0 $ and $m(\xi)\rightarrow \infty$ as $\xi\rightarrow \infty$, then $\tilde{m}=\frac{1}{m}$ is bounded and $\lim_{\xi\rightarrow \pm \infty}\tilde{m}(\xi)=0$. Moreover, 
\begin{equation*}
\frac{1}{c-m}=\frac{1}{m}\frac{1}{c/m-1}=-\frac{1}{c}\frac{\tilde{m}}{\frac{1}{c}-\tilde{m}}.
\end{equation*}
Hence, letting $H_c$ be defined by \eqref{eq: H_c} as in Sections \ref{sec: H_c} and \ref{sec: decay}, with $\tilde{m}$ in place of $m$, that is, $H_c=\F^{-1}\left(\frac{\tilde{m}}{c-\tilde{m}}\right)$, we get that \eqref{eq: main} can be written as
\begin{equation}
cu= -H_{1/c}\ast G(u).
\end{equation}
Note that $c$ is just a constant and the minus sign makes no difference to our results as we have no assumptions on the sign of $G$. Hence this equation is even simpler than \eqref{eq: main eq general}, as we do not have the term $c-\frac{G(u)}{u}$ on the left-hand side, and all our results are therefore valid if assumptions (A1), (A2) and (A3) (and (A1*)) are satisfied for $\tilde{m}$, $\frac{1}{c}$ and $G$. We summarize the results in the following theorem:

\begin{theorem}
	\label{thm: diff operator}
	Let (A3) be satisfied, $m\colon \R \rightarrow \R$ be even and strictly positive, and such that $m^{-1}$ satisfies (A1), and let $0<c<\min_{\xi\in\R}m(\xi)$. Suppose that $u\in L^\infty(\R)$ with $\lim_{|x|\rightarrow \infty}u(x)=0$ is a non-trivial solution to \eqref{eq: main}. Then
	\begin{equation*}
	|\cdot|^l u(\cdot)\in L^\infty(\R),
	\end{equation*}
	for any $l\geq 0$.
	If $m^{-1}$ satisfies (A1*) in addition, then 
	\begin{equation*}
	\mathrm{e}^{\delta_{1/c}|\cdot|}u(\cdot)\in L^p(\R), \,\, \text{if and only if} \,\, p=\infty,
	\end{equation*}
	where $\delta_{1/c}>0$ is the smallest number for which there exists a $z_0\in \C$ with $\mathrm{Im}\, z_0=\delta_{1/c}$ such that $m(z_0)^{-1}=\frac{1}{c}$.
\end{theorem}

\section{Symmetry of solitary waves}
\label{sec: symmetry}
Now we will prove Theorem \ref{thm: symmetry}. The method is based on the method of moving planes and is an adaption of the proof the same result for solitary waves to the Whitham equation done in \cite{Bruell2017sad}.

Essential to the proof of symmetry is the following "touching" lemma:
\begin{lemma}
	\label{lem: touching}
	Let $H_c$ be as in Theorem \ref{thm: symmetry}, and let $u\in L^\infty(\R)$ with $\lim_{|x|\rightarrow \infty} u(x)=0$ be a solution to \eqref{eq: main eq general} and assume that $G$ is non-negative and increasing on the range of $u$. Denote by $u_\lambda (\cdot):=u(2\lambda-\cdot)$ the reflection of $u$ about $\lambda\in \R$. If $u\geq u_\lambda$ on $[\lambda,\infty)$, then either
	\begin{itemize}
		\item $u=u_\lambda$, or
		\item $u>u_\lambda$ and $\frac{G(u)}{u}+\frac{G(u_\lambda)}{u_\lambda}<c$ for all $x>\lambda$.
	\end{itemize}
	That is, if $u\geq u_\lambda$ on $(\lambda,\infty)$, then either they are equal or they do not touch.
\end{lemma}
This lemma is essentially corollary 4.2 in \cite{Bruell2017sad} for a general class of equations and can be proved in a similar manner. For completeness we include the proof; some of the arguments will also be useful later.

\begin{proof}
	Let $f\geq 0$ on $[\lambda,\infty)$ be odd about $\lambda$, that is, $f(x)=-f(2\lambda-x)$, and let $x\geq \lambda$. A simple change of variables and that $f$ is odd with respect to $\lambda$ gives that
	\begin{align*}
	H_c\ast f(x)=& \int_\lambda^\infty H_c(x-y)f(y)\dy+\int_{-\infty}^\lambda H_c(x-y)f(y)\dy \\
	= & \int_\lambda^\infty H_c(x-y)f(y)\dy+\int_\lambda^\infty H_c(x+y-2\lambda)f(2\lambda -y)\dy \\
	= & \int_\lambda^\infty \left(H_c(x-y)-H_c(x+y-2\lambda)\right)f(y)\dy.
	\end{align*}
	As $H_c$ is symmetric and monotonically decreasing on $(0,\infty)$, and $f\geq 0$ on $[\lambda,\infty)$, we conclude that
	\begin{equation*}
	H_c\ast f(x)\geq 0 \,\, \text{for all} \,\, x\geq \lambda,
	\end{equation*}
	with equality if and only if $f=0$ on $(\lambda,\infty)$. By the definition of $u_\lambda$, $G(u)-G(u_\lambda)$ is odd about $\lambda$, and as $u(x)\geq u_l(x)$ for $x\geq \lambda$, it follows from the assumption that $G$ is increasing on the range of $u$ that $G(u)-G(u_\lambda)\geq 0$ for $x\geq \lambda$. Hence $G$ satisfies the same properties as $f$, and by the symmetry of $H_c$ we have that $u_\lambda$ is also a solution to \eqref{eq: main eq general}. We therefore conclude that
	\begin{equation*}
	(u-u_\lambda)\left(c-\frac{G(u)}{u}-\frac{G(u_\lambda)}{u_\lambda}\right)=H_c\ast\left(G(u)-G(u_\lambda)\right)>0
	\end{equation*}
	for all $x>\lambda$ unless $u=u_\lambda$.
\end{proof}

With this result we can prove Theorem \ref{thm: symmetry}:

\begin{proof}[Proof of Theorem \ref{thm: symmetry}]
	Following \cite{Chen2006cos}, we define
	\begin{equation*}
	\Sigma_\lambda:=\lbrace x\in \R : x>\lambda \rbrace
	\end{equation*}
	and
	\begin{equation*}
	\Sigma_\lambda^- :=\lbrace x\in \Sigma_\lambda : u(x)<u_\lambda(x)\rbrace.
	\end{equation*}
	The first step is to show that there is a $\lambda$ far enough to the left such that the open set $\Sigma_\lambda^-$ is empty. A straightforward calculation similar to the one in Lemma \ref{lem: touching} gives that
	\begin{align*}
	c& (u(x)-u_\lambda(x)) \\
	&=\int_{\Sigma_\lambda} \left(H_c(x-y)-H_c(x+y-2\lambda)\right)\left( G(u(y))-G(u_\lambda(y))\right)\dy + G(u(x))-G(u_\lambda(x)).
	\end{align*}
	Let $x\in \Sigma_\lambda^-$ and let $r>1$ be as in assumption (A3). Then
	\begin{align*}
	0<&c(u_\lambda(x)-u(x)) \\
	\leq & \int_{\Sigma_\lambda^-} \left(H_c(x-y)-H_c(2\lambda-x-y)\right)\left( G(u_\lambda(y))-G(u(y))\right)\dy + G(u_\lambda(x))-G(u(x)) \\
	\leq & \int_{\Sigma_\lambda^-}H_c(x-y)u_\lambda(y)^{r-1}\left( \frac{G(u_\lambda(y))}{u_\lambda(y)^{r-1}}-\frac{G(u(y))}{u(y)^{r-1}}\right)\dy + G(u_\lambda(x))-G(u(x)).
	\end{align*}
	By H\"older's inequality we get that
	\begin{equation}
	\label{eq: symmetry eq}
	\|u_\lambda-u\|_{L^\infty(\Sigma_\lambda^-)}\leq \frac{1}{c}\|u_\lambda\|_{L^\infty(\Sigma_\lambda^-)}^{r-1}\left(\|H_c\|_{L^1(\R)}+1\right)\|\frac{G(u_\lambda)}{u_\lambda^{r-1}}-\frac{G(u)}{u^{r-1}}\|_{L^\infty(\Sigma_\lambda^-)}.
	\end{equation}
	Note that every term on the right-hand side is bounded independently of $\lambda$. Moreover, as $u_\lambda(x)=u(2\lambda-x)$ and $u$ is decaying, we get that $\lim_{\lambda\rightarrow -\infty} \|u_\lambda\|_{L^\infty(\Sigma_\lambda^-)}=0$, which implies that $\|u_\lambda-u\|_{L^\infty(\Sigma_\lambda^-)}\rightarrow 0$ as well. Hence, for $\lambda\ll 0$ sufficiently small, assumption (A3) gives that
	\begin{equation*}
	\|\frac{G(u_\lambda)}{u_\lambda^{r-1}}-\frac{G(u)}{u^{r-1}}\|_{L^\infty(\Sigma_\lambda^-)}\leq C\|u_\lambda-u\|_{L^\infty(\Sigma_\lambda^-)},
	\end{equation*}
	for some $C$. As long as $\|u_\lambda-u\|_{L^\infty(\Sigma_\lambda^-)}\neq 0$, we can divide by this term on both sides in \eqref{eq: symmetry eq}, and we see that there must exists an $N\in \R$, such that $\|u_\lambda-u\|_{L^\infty(\Sigma_\lambda^-)}=0$ for all $\lambda\leq-N$. It follows that $\Sigma_\lambda^-$ is empty for all $\lambda\leq -N$, and that $u$ cannot have any crests to the left of $-N$.
	
	The next step now is to move the plane $x=\lambda$ to the right from $\lambda=-N$ until the final point for which $\Sigma_\lambda^-$ is empty. This process will stop at a crest of before. Assume the process stops at a point $\lambda_0$, where $u(x)\geq u_\lambda(x)$, but $u(x)\neq u_{\lambda_0}(x)$ for all $x\in\Sigma_{\lambda_0}$. That is, $u$ is not symmetric about $\lambda_0$. By Lemma \ref{lem: touching}, we get that $u(x)>u_{\lambda_0}(x)$ for all $x\in \Sigma_{\lambda_0}$. As $u$ is continuous, we have that for any $\varepsilon>0$, there is a $\delta>0$ such that $|\overline{\Sigma_\lambda^-}|<\varepsilon$ for all $\lambda\in [\lambda_0,\lambda_0+\delta)$. Let $\lambda>\lambda_0$ with $|\lambda-\lambda_0|$ sufficiently small such that $\Sigma_\lambda^-$ is bounded (by assumption, $\Sigma_\lambda^-$ is non-empty, otherwise the process of the plane would not have stopped at $\lambda_0$). Let $x\in \Sigma_\lambda^-$. By similar calculations as those preceding \eqref{eq: symmetry eq}, we get that
	\begin{align*}
	0<&c(u_\lambda(x)-u(x)) \\
	\leq & \int_{\Sigma_\lambda^-}H_c(x-y)(u_\lambda(y)-u(y))\left( \frac{G(u_\lambda(y))-G(u(y))}{u_\lambda(y)-u(y)}\right)\dy + G(u_\lambda(x))-G(u(x)).
	\end{align*}
	Let $p\in (1,\infty)$. By Young's and H\"older's inequalities, we get that
	\begin{align*}
	c\|u_\lambda-u\|_{L^p(\Sigma_\lambda^-)}\leq & \|H_c\|_{L^s(\R)}\|(u_\lambda-u)\left( \frac{G(u_\lambda)-G(u)}{u_\lambda-u}\right)\|_{L^q(\Sigma_\lambda^-)}\\
	& +\|G(u_\lambda)-G(u)\|_{L^p(\Sigma_\lambda^-)} \\
	\leq & \|H_c\|_{L^s(\R)}\| \frac{G(u_\lambda)-G(u)}{u_\lambda-u}\|_{L^{qp/(p-q)}(\Sigma_\lambda^-)} \|u_\lambda-u\|_{L^p(\Sigma_\lambda^-)}\\
	& + \| \frac{G(u_\lambda)-G(u)}{u_\lambda-u}\|_{L^\infty(\Sigma_\lambda^-)}\|u_\lambda-u\|_{L^p(\Sigma_\lambda^-)},
	\end{align*}
	where $s,q\in [1,\infty)$ are chosen such that $1+\frac{1}{p}=\frac{1}{s}+\frac{1}{q}$. Note that this choice can be made such that $q>p$ and hence $1<\frac{qp}{p-q}<\infty$. Since $\Sigma_\lambda^-$ is assumed to be non-empty, the continuity of $u$ implies that $\|u_\lambda-u\|_{L^p(\Sigma_\lambda^-)}>0$, so we can divide out this term and we get that
	\begin{equation}
	\label{eq: symmetry eq 2}
	c\leq \|H_c\|_{L^s(\R)}\| \frac{G(u_\lambda)-G(u)}{u_\lambda-u}\|_{L^{qp/(p-q)}(\Sigma_\lambda^-)}+\| \frac{G(u_\lambda)-G(u)}{u_\lambda-u}\|_{L^\infty(\Sigma_\lambda^-)}.
	\end{equation}
	As $|\overline{\Sigma_\lambda^-}|\rightarrow 0$ as $\lambda\rightarrow \lambda_0^+$, the first term on the right-hand side can be made arbitrarily small by taking $\lambda>\lambda_0$ close enough to $\lambda_0$. By assumption we have that $G(u_\lambda)-G(u)\leq \tilde{c}(u_\lambda-u)$, so that
	\begin{equation*}
	\| \frac{G(u_\lambda)-G(u)}{u_\lambda-u}\|_{L^\infty(\Sigma_\lambda^-)}\leq \tilde{c}<c.
	\end{equation*}
	We have thus showed that there is a $\delta>0$ such that the right-hand side of \eqref{eq: symmetry eq 2} is less than $c$ for all $\lambda\in [\lambda_0,\lambda_0+\delta)$, which is clearly a contradiction. Hence it must be the case that $\|u_\lambda-u\|_{L^p(\Sigma_\lambda^-)}=0$, which implies that $\Sigma_\lambda^-$ is empty - a contradiction. It follows that the assumption that $u$ is not symmetric about $\lambda_0$ is false and this completes the proof.
\end{proof}

\section{Examples}
\label{sec: examples}
In this section we apply our theory from the preceding sections to some equations of interest, for which the (precise) decay properties have not previously been established.

\subsection*{A Whitham--Boussinesq system}
Let us return to the Whitham--Boussinesq system mentioned in the introduction (cf. \eqref{eq: whitham-boussinesq}). Solitary-wave solutions to this system satisfy (see \eqref{eq: main eq})
\begin{equation*}
u\left(c^2-\frac{G(u)}{u}\right)=H_c\ast G(u),
\end{equation*}
where $G(u)=\frac{u^2}{2}(3c-u)$ and
\begin{equation}
	\label{eq: H_c for whitham-boussinesq}
	H_c=\F^{-1}\left(\frac{m}{c^2-m}\right).
\end{equation}
This is exactly of the form \eqref{eq: main eq general} only with $c$ replaced by $c^2$. Clearly, $G$ satisfies (A3) with $r=2$ and hence, if (A2) is satisfied with $c^2$ in place of $c$, all the results of the previous sections are valid. 
A specific equation of particular interest within this class is when $m$ is the bi-directional Whitham-Kernel:
\begin{equation}
\label{eq: bi-directional whitham}
	m(\xi)=\frac{\tanh(\xi)}{\xi}.
\end{equation}
For this $m$, the theory in the previous sections gives the following result:
\begin{theorem}
	Let $c>1$, $m(\xi)=\frac{\tanh(\xi)}{\xi}$, and $\delta_c\in (0,\frac{\pi}{2})$ satisfy $\frac{\tan(\delta_c)}{\delta_c}=c^2$. Then, for $H_c$ defined as in \eqref{eq: H_c for whitham-boussinesq},
	\begin{equation*}
		H_c(x)=\mathrm{e}^{-\delta_c|x|}\left(v(x)+\sqrt{2\pi}\frac{\tan(\delta_c)\delta_c}{\delta_c\sec^2(\delta_c)-\tan(\delta_c)}\right),
	\end{equation*}
	for some even function $v$ that satisfies $v\in L^p(\lbrace x\in \R : |x|\geq 1\rbrace)$ for all $1\leq p\leq \infty$, and $v(x)\simeq |\ln(|x|)|$ for $|x|\ll 1$.
	
	Moreover, if $u\in L^\infty(\R)$ with $\lim_{|x|\rightarrow \infty} u(x)=0$ is a non-trivial solution to \eqref{eq: main eq}, then
	\begin{equation*}
	u(\cdot)\mathrm{e}^{\delta_c|\cdot|}\in L^p(\R) \,\, \text{if and only if} \,\, p=\infty.
	\end{equation*}
	That is, $u(x)$ decays exactly like $\mathrm{e}^{-\delta_c|x|}$. Furthermore, if $u \colon \R \rightarrow [0,c-\frac{1}{3}]$, then $u$ is symmetric.
\end{theorem}

\begin{proof}
	It is straightforward to see that $m$ satisfies (A1) and (A2) with $m_0=-1$, so all results in Section \ref{sec: H_c} hold in the present case; in particular $H_c(x)\simeq |\ln(|x|)|$ for $x$ near $0$ (cf. Lemma \ref{lem: behaviour at 0}). Moreover, $m(z)$ is analytic for all $z\in \C$ except $z\in i\frac{\pi}{2} \Z\setminus \lbrace 0\rbrace$ and as $m$ is even and monotonically decreasing on $(0,\infty)$, it is real-valued on, and only on, the real and the imaginary axis. Along the imaginary axis,
	\begin{equation*}
	m(iy)=\frac{\tan(y)}{y},
	\end{equation*}
	which is even in $y$ and a bijection from $[0,\frac{\pi}{2})$ to $[1,\infty)$. Hence, for all $c>1$, the equation $\frac{\tan(y)}{y}=c^2$ has one solution in $(0,\frac{\pi}{2})$, which we denote by $\delta_c$. By Lemma \ref{lem: exponential decay kernel L2} we get that 
	\begin{equation*}
		H_c(x)=\mathrm{e}^{-\delta_c|x|}\left(v(x) +C\right),
	\end{equation*}
	with $v$ as in the statement and some $C$. As the singularities are at $\pm i \delta_c$, we can use \eqref{eq: expression for H_c} to calculate $C$ explicitly in terms of $c$ (recall that $\frac{\tan(\delta_c)}{\delta_c}=c^2$) and we get the expression in the statement.
	
	With the expression for $H_c$, the decay of $u$ follows directly from Theorem \ref{thm: decay}. It remains only to show symmetry. It is straightforward to check that the function $G(x)=\frac{x^2}{2}(3c-x)$ is increasing on $[0,2c]$ and $G'(x)<c^2$ on $[0,c-\frac{1}{\sqrt{3}})$, and that $H_c$ satisfies the assumptions in Theorem \ref{thm: symmetry} (cf. Remark \ref{rem: assumptions}). The symmetry then follows from Theorem \ref{thm: symmetry}.
\end{proof}

\subsection*{The Whitham equation}

Let us now turn to the Whitham equation
\begin{equation}
\label{eq: whitham general}
	u_t+2uu_x+Lu_x=0,
\end{equation}
where, $m(\xi)=\sqrt{\frac{\tanh(\xi)}{\xi}}$. In this case solitary wave solutions will satisfy the equation
\begin{equation}
\label{eq: whitham}
u(c-u)=H_c \ast u^2.
\end{equation}
Clearly $m$ satisfies (A1) and (A2) with $m_0=-1/2$.

In \cite{Bruell2017sad} they prove that for $c>1$
\begin{equation*}
	\mathrm{e}^{\delta |\cdot|} (\cdot)H_c(\cdot)\in L^2(\R), \,\, \text{for any} \,\, \delta\in (0, \delta_c),
\end{equation*}
where $\delta_c\in (0, \frac{\pi}{2})$ satisfies $\sqrt{\frac{\tan(\delta_c)}{\delta_c}}=c$, without showing whether or not this is optimal. Moreover, they prove that solitary waves satisfy
\begin{equation*}
	\mathrm{e}^{\eta |\cdot|}u(\cdot)\in L^1(\R)\cap L^\infty(\R), \,\, \text{for some} \,\, \eta\geq \delta.
\end{equation*}
With our results from Sections \ref{sec: H_c} and \ref{sec: decay}, we can improve upon these results by giving the precise rate of decay both for the kernel $H_c$ and for a solitary-wave solution $u$:
\begin{theorem}
	\label{thm: decay whitham}
	Let $c>1$, $m(\xi)=\sqrt{\frac{\tanh(\xi)}{\xi}}$, and $\delta_c\in (0, \frac{\pi}{2})$ satisfy $\sqrt{\frac{\tan(\delta_c)}{\delta_c}}=c$. Then
	\begin{equation*}
		H_c(x)=\mathrm{e}^{-\delta_c|x|}\left(v(x)+\sqrt{2\pi}\frac{2\tan(\delta_c)\delta_c}{\delta_c\sec^2(\delta_c)-\tan(\delta_c)}\right),
	\end{equation*}
	for some even function $v\in L^p(\lbrace x\in \R: |x|\geq 1\rbrace)$ for all $1\leq p\leq \infty$ that satisfies $v(x)\simeq |x|^{-1/2}$ for $|x|<1$.
	
	Moreover, if $u\in L^\infty(\R)$ with $\lim_{|x|\rightarrow \infty} u(x)=0$ is a non-trivial solution to \eqref{eq: whitham}, then
	\begin{equation*}
		\mathrm{e}^{\delta_c|\cdot|}u(\cdot)\in L^p(\R) \,\, \text{if and only if} \,\, p=\infty.
	\end{equation*}
\end{theorem}

\begin{proof}
	As noted above, $m$ satisfies (A1) and (A2) with $m_0=-\frac{1}{2}$, so all results in Section \ref{sec: H_c} hold in the present case; in particular $H_c(x)\simeq |x|^{-1/2}$ for $x$ near $0$ (cf. Lemma \ref{lem: behaviour at 0}). Moreover, $m(z)^2$ is analytic for all $z\in \C$ except $z\in i\frac{\pi}{2} \Z\setminus \lbrace 0\rbrace$. Hence $m(z)$ is analytic in the strip $|\mathrm{Im}\, z|<\frac{\pi}{2}$. Moreover, as $m$ is even and, clearly, monotonically decreasing on $(0,\infty)$, it is real-valued on, and only on, the real and the imaginary axis. Along the imaginary axis,
	\begin{equation*}
	m(iy)=\sqrt{\frac{\tan(y)}{y}},
	\end{equation*}
	which is even in $y$ and a bijection from $[0,\frac{\pi}{2})$ to $[1,\infty)$. Hence, for all $c>1$, the equation $\sqrt{\frac{\tan(y)}{y}}=c$ has one solution in $(0,\frac{\pi}{2})$, which we denote by $\delta_c$, and $g$ has two singularities within the strip $|\mathrm{Im}\, z|<\frac{\pi}{2}$, namely at $\pm i\delta_c$. It follows from Lemma \ref{lem: exponential decay kernel L2}
	\begin{equation*}
	H_c(x)=\mathrm{e}^{-\delta_c|x|}\left(v(x)+C\right),
	\end{equation*}
	for $v$ as in the statement and some $C$. However, as our singularities are at $\pm i \delta_c$, we can use \eqref{eq: expression for H_c} to calculate
	\begin{equation*}
	C=-i\sqrt{2\pi}\frac{c}{m'(i\delta_c)}=\sqrt{2\pi}\frac{2\tan(\delta_c)\delta_c}{\delta_c\sec^2(\delta_c)-\tan(\delta_c)},
	\end{equation*}
	where we used that $c=\sqrt{\frac{\tan(\delta_c)}{\delta_c}}$. This proves the first part.
	
	For the second part, note that $G(u)=u^2$ satisfies (A3) with $r=2$. Having proved the first part, the second part now follows by Theorem \ref{thm: decay}.
\end{proof}

%Consider the solitary-wave solutions for the Whitham equation found in \cite{EGW} that depend on a small parameter $\mu>0$. These have wave speeds
%\begin{equation}
%\label{eq: c and mu}
%	c=c(u)=1+\mu^{2/3}\left(\frac{2}{3}\right)^{1/3}+o(\mu^{2/3}),
%\end{equation}
%and are scalings of the classical KdV solitary wave in the following sense:
%\begin{equation*}
%	\inf_{y\in \R}\|\mu^{-2/3}u(\mu^{-1/3}(\cdot+y))-\left(\frac{3}{2}\right)^{2/3}\sech^2\left(\left(\frac{3}{2}\right)^{1/3}\cdot\right)\|_{H^1(\R)}\rightarrow 0, \,\, \text{as}\,\, \mu \searrow 0.
%\end{equation*}
%That is, $u$ is close in $H^1(\R)$ to something that decays like $\mathrm{e}^{-2\left(\frac{3}{2}\right)^{1/3}\mu^{1/3}|\cdot|}$. From \eqref{eq: c and mu} we get that
%\begin{equation*}
%	2\left(\frac{3}{2}\right)^{1/3}\mu^{1/3}=\sqrt{6}(c-1)^{1/2}+o(\mu^{1/3}).
%\end{equation*}
%Recalling that $c=\sqrt{\frac{\tan(\delta_c)}{\delta_c}}=1+\frac{1}{6}\delta_c^2+O(\delta_c^4)$, we get the relation
%\begin{equation}
%	\delta_c=2\left(\frac{3}{2}\right)^{1/3}\mu^{1/3}+o(\mu^{1/3}).
%\end{equation}
%Hence, for $c>1$ sufficiently close to $1$, the "decay rate" from \cite{EGW} agrees with that we found in Theorem \ref{thm: decay whitham} above.

\subsection*{The Capillary Whitham equation}
The examples above were with very weak dispersion, but as shown in Section \ref{subsec: diff operator} the theory can also be applied to equations with stronger dispersion. We take the Capillary Whitham equation as an example. That is, we consider \eqref{eq: whitham general}, now with
\begin{equation}
	\label{eq: capillary whitham kernel}
	m(\xi)=\sqrt{\frac{(1+\beta \xi^2)\tanh(\xi)}{\xi}},
\end{equation}
where $\beta>0$, called the Bond number, is the strength of the surface tension. In this case we have that all sub-critical solitary wave solutions are exponentially decaying:
\begin{theorem}
	Let $\beta>0$ and $m$ be defined by \eqref{eq: capillary whitham kernel}, and let $0<c<\min_{\xi\in\R} m(\xi)$. Denoting by $\delta_{1/c}>0$ the smallest positive number for which there exists a $z_0\in \C$ with $\mathrm{Im} \, z_0=\delta_{1/c}$ such that $m(z_0)=c$, we have that if $u\in L^\infty(\R)$ with $\lim_{|x|\rightarrow  \infty} u(x)=0$ is a non-trivial solution to \eqref{eq: whitham}, then
	\begin{equation*}
		\mathrm{e}^{\delta_{1/c}|\cdot|}u(\cdot)\in L^p(\R) \,\, \text{if and only if} \,\, p=\infty.
	\end{equation*}
\end{theorem}
Noting that (A3) is clearly satisfied and for all $\beta>0$ and $0<c<\min_{\xi\in\R} m(\xi)$,
\begin{equation*}
	\tilde{m}(\xi):=\frac{1}{m(\xi)}=\sqrt{\frac{\xi}{(1+\beta\xi^2)\tanh(\xi)}}
\end{equation*}
satisfies (A1), (A1*) and (A2) (with $\frac{1}{c}$ in place of $c$), the result follows directly from Theorem \ref{thm: diff operator}. However, it is still of interest to investigate some of the dynamics. Let $z_0\in C$ with $\mathrm{Im}\, z_0= \delta_{1/c}$ be such that $\tilde{m}=\frac{1}{c}$. We observe that $\tilde{m}$ is analytic in
\begin{equation*}
	\C\setminus \lbrace iy : y\in \R\setminus \lbrace 0\rbrace, \,\, \sign(y)(1-\beta y^2)\tan(y)\leq 0\rbrace,
\end{equation*}
and the intervals cut out from the imaginary axis are branch cuts. In particular it is analytic in the strip $|\mathrm{Im} \, z|< \min(\sqrt{\beta^{-1}},\frac{\pi}{2})$. If $\beta>\frac{4}{\pi^2}$, then
\begin{equation*}
	\tilde{m}(iy)=\sqrt{\frac{y}{(1-\beta y^2)\tan(y)}}\colon [0,\beta^{-1/2}) \rightarrow [1,\infty),
\end{equation*}
is a bijection; in particular, $z_0$ lies on the imaginary axis within the strip where $\tilde{m}$ is analytic. If $\beta<\frac{4}{\pi}$, however, then the point $z_0$ does not lie within the strip, and not necessarily even on the imaginary axis (if $z_0$ is not purely imaginary, then $\tilde{m}(-\overline{z_0})=\frac{1}{c}$ as well).

\medskip
\noindent

\cleardoublepage
\bibliographystyle{plain}
\bibliography{decay.bib}

\end{document}